\documentclass[12pt]{amsart}
\usepackage{amsmath,amsfonts,amsthm,amssymb,bbm,tikz,url}
\usetikzlibrary{decorations.pathmorphing}
\usepackage[margin=1.in]{geometry}

\newcommand{\inv}{\mathrm{inv}}
\newcommand{\semilength}{\mathrm{semilength}}
\newcommand{\height}{\mathrm{height}}
\newcommand{\swb}{\mathrm{swb}}
\newcommand{\neb}{\mathrm{neb}}
\newcommand{\nesting}{\mathrm{N}}
\newcommand{\FPL}[1]{\textrm{FPL}_#1}
\newcommand{\C}{\mathrm{C}}
\newcommand{\diag}{\mathrm{Diag}}

\newtheorem{thm}{Theorem}[section]

\newtheorem{lemma}[thm]{Lemma}
\newtheorem{prop}[thm]{Proposition}

\theoremstyle{definition}

\newtheorem{defn}[thm]{Definition}

\begin{document}
\title{Alternating Paths of Fully Packed Loops and Inversion Number}

\author[S. Ng]{Stephen Ng}
\address{Department of Mathematics\\
University of Rochester\\
Rochester, NY 14627, USA}
\email{ng@math.rochester.edu}

\date{Version: \today}
\maketitle

\begin{abstract}
We consider the set of alternating paths on a fixed fully packed loop of size $n$, which we denote by $\phi_0$. This set is in bijection with the set of fully packed loops of size $n$ and is also in bijection with the set of alternating sign matrices by a well known bijection. Furthermore, for a special choice of $\phi_0$, we demonstrate that the set of alternating paths are nested osculating loops, which give rise to a modified height function representation which we call Dyck islands. Dyck islands can be constructed as a union of lattice Dyck paths, and we use this structure to give a simple graphical formula for the calculation of the inversion number of an alternating sign matrix.
\end{abstract}

\section{Introduction}
The motivation for studying alternating paths of fully packed loops began with the online note of Ayyer and Zeilberger \cite{AZ} on an attempt to prove the Razumov-Stroganov conjecture. This note introduced the notion of an alternating path of a fully packed loop and described their action on the underlying link patterns in simple example cases. Furthermore, they conjectured the existence of an algorithm for finding an alternating path that would implement the pullback of the local XXZ Hamiltonians into the space of fully packed loops in such a way that would provide a solution to the Razumov-Stroganov conjecture. The RS conjecture has since been solved by Cantini and Sportiello's detailed analysis \cite{CantiniSportiello} of Wieland's gyration operation \cite{Wieland} on fully packed loops, but the question of the existence of an algorithm with the desired properties remains open.

On another note, Striker \cite{StrikerPolytope} and Behrend and Knight \cite{BK} independently studied the notion of the alternating sign matrix polytope. In particular, Striker  gave a nice characterization of the face lattice of this polytope in terms of what she called doubly directed regions of flow diagrams \cite{StrikerPolytope}. 

Recast into the fully packed loop picture, this is described as follows: Given any two fully packed loops, there is an alternating path (possibly a disjoint union of alternating loops) along which they differ in color. Then given some collection of fully packed loops of size $n$, one can consider the union of all alternating paths between pairs of fully packed loops. This union represents the smallest face of the alternating sign matrix polytope which contains all of the fully packed loops in the collection.

In what follows, this paper is divided into two additional sections. In Section \ref{sec:defn}, we present all relevant definitions and develop the correspondence between alternating sign matrices, fully packed loops, and Dyck islands. Section \ref{sec:inv} of this paper then demonstrates the utility of this new representation by establishing a connection between the shape of the Dyck island and the inversion number of an alternating sign matrix. In particular, we show that the inversion number of an alternating sign matrix can be decomposed as follows:
\begin{equation*}
	\inv(A) = \sum_{i=1}^{\ell} \inv(\gamma_i) - k
\end{equation*}
where $\gamma_1, \ldots, \gamma_\ell$ are boundary paths of the Dyck island corresponding to $A$, $k$ is the number of off-diagonal osculations of these paths, and $\inv(\gamma_i)$ is the inversion number of the alternating sign matrix corresponding to the Dyck island described by just $\gamma_i$. The quantity $\inv(\gamma_i)$ will be shown to be dependent only on the diameter of the loop and the number of osculations on the diagonal. See Figure \ref{fig:DIinv} below for some preliminary examples.

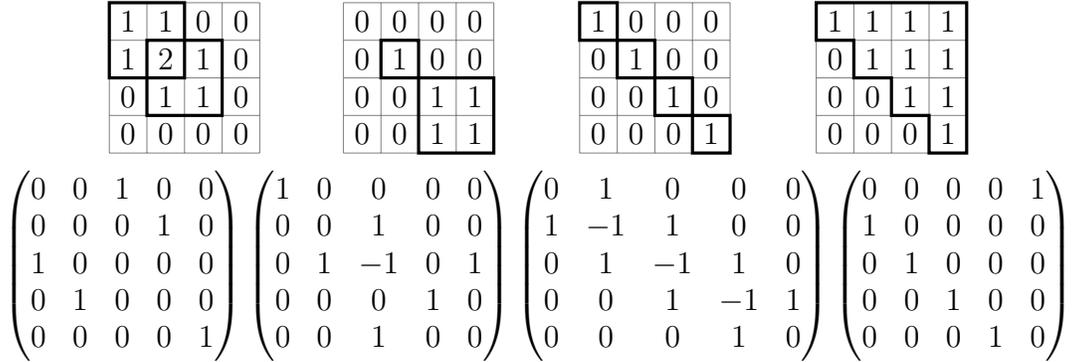
\begin{figure}[h]
	\begin{center}
		\begin{tikzpicture}
			\draw[step=.5cm,gray,very thin] (-1,-1) grid (1,1);
			\draw[very thick] (-1,1) -- (0,1) -- (0,0) -- (-1,0) -- cycle;
			\draw[very thick] (-.5,.5) -- (.5,.5) -- (.5,-.5) -- (-.5,-.5) -- cycle;
			\node at (-.75,.75) {$1$};
			\node at (-.25,.75) {$1$};
			\node at (.25,.75) {$0$};
			\node at (.75,.75) {$0$};
			\node at (-.75,.25) {$1$};
			\node at (-.25,.25) {$2$};
			\node at (.25,.25) {$1$};
			\node at (.75,.25) {$0$};
			\node at (-.75,-.25) {$0$};
			\node at (-.25,-.25) {$1$};
			\node at (.25,-.25) {$1$};
			\node at (.75,-.25) {$0$};
			\node at (-.75,-.75) {$0$};
			\node at (-.25,-.75) {$0$};
			\node at (.25,.-.75) {$0$};
			\node at (.75,-.75) {$0$};
		\end{tikzpicture}
		$\qquad$
		\begin{tikzpicture}
			\draw[step=.5cm,gray,very thin] (-1,-1) grid (1,1);
			\draw[very thick] (-.5,0) rectangle (0,.5);
			\draw[very thick] (0,-1) rectangle (1,0);
			\node at (-.75,.75) {$0$};
			\node at (-.25,.75) {$0$};
			\node at (.25,.75) {$0$};
			\node at (.75,.75) {$0$};
			\node at (-.75,.25) {$0$};
			\node at (-.25,.25) {$1$};
			\node at (.25,.25) {$0$};
			\node at (.75,.25) {$0$};
			\node at (-.75,-.25) {$0$};
			\node at (-.25,-.25) {$0$};
			\node at (.25,-.25) {$1$};
			\node at (.75,-.25) {$1$};
			\node at (-.75,-.75) {$0$};
			\node at (-.25,-.75) {$0$};
			\node at (.25,.-.75) {$1$};
			\node at (.75,-.75) {$1$};
		\end{tikzpicture}
		$\qquad$
		\begin{tikzpicture}
			\draw[step=.5cm,gray,very thin] (-1,-1) grid (1,1);
			\draw[very thick] (-1,.5) rectangle (-.5,1);
			\draw[very thick] (-.5,0) rectangle (0,.5);
			\draw[very thick] (0,-.5) rectangle (.5,0);
			\draw[very thick] (.5,-1) rectangle (1,-.5);
			\node at (-.75,.75) {$1$};
			\node at (-.25,.75) {$0$};
			\node at (.25,.75) {$0$};
			\node at (.75,.75) {$0$};
			\node at (-.75,.25) {$0$};
			\node at (-.25,.25) {$1$};
			\node at (.25,.25) {$0$};
			\node at (.75,.25) {$0$};
			\node at (-.75,-.25) {$0$};
			\node at (-.25,-.25) {$0$};
			\node at (.25,-.25) {$1$};
			\node at (.75,-.25) {$0$};
			\node at (-.75,-.75) {$0$};
			\node at (-.25,-.75) {$0$};
			\node at (.25,.-.75) {$0$};
			\node at (.75,-.75) {$1$};
		\end{tikzpicture}
		$\qquad$
		\begin{tikzpicture}
			\draw[step=.5cm,gray,very thin] (-1,-1) grid (1,1);
			\draw[very thick] (1,-1) -- ++(-.5,0) -- ++(0,.5) -- ++(-.5,0) -- ++(0,.5) -- ++(-.5,0) -- ++(0,.5) -- ++(-.5,0) -- ++(0,.5) -- ++(2,0) -- cycle;
			\node at (-.75,.75) {$1$};
			\node at (-.25,.75) {$1$};
			\node at (.25,.75) {$1$};
			\node at (.75,.75) {$1$};
			\node at (-.75,.25) {$0$};
			\node at (-.25,.25) {$1$};
			\node at (.25,.25) {$1$};
			\node at (.75,.25) {$1$};
			\node at (-.75,-.25) {$0$};
			\node at (-.25,-.25) {$0$};
			\node at (.25,-.25) {$1$};
			\node at (.75,-.25) {$1$};
			\node at (-.75,-.75) {$0$};
			\node at (-.25,-.75) {$0$};
			\node at (.25,.-.75) {$0$};
			\node at (.75,-.75) {$1$};
		\end{tikzpicture}
		\begin{equation*}
			\begin{pmatrix}
				0 & 0 & 1 & 0 & 0 \\
				0 & 0 & 0 & 1 & 0 \\
				1 & 0 & 0 & 0 & 0 \\
				0& 1 & 0 & 0 & 0 \\
				0& 0 & 0 & 0 & 1
			\end{pmatrix}
			\begin{pmatrix}
				1 & 0 & 0 & 0 & 0 \\
				0 & 0 & 1 & 0 & 0 \\
				0 & 1 & -1 & 0 & 1 \\
				0 & 0 & 0 & 1 & 0 \\
				0 & 0 & 1 & 0 & 0
			\end{pmatrix}
			\begin{pmatrix}
				0 & 1 & 0 & 0 & 0 \\
				1 & -1 & 1 & 0 & 0\\
				0 & 1 & -1 & 1 & 0 \\
				0& 0 & 1 & -1 & 1 \\
				0 & 0 & 0 & 1 & 0
			\end{pmatrix}
			\begin{pmatrix}
				0 & 0 & 0& 0 & 1 \\
				1 & 0 & 0 & 0 & 0 \\
				0 & 1 & 0 & 0 & 0 \\
				0 & 0 & 1 & 0 & 0 \\
				0 & 0 & 0 & 1 & 0 
			\end{pmatrix}
		\end{equation*}
	\end{center}
	\caption{Some example Dyck islands with inversion number 4 and their corresponding alternating sign matrices.}
	\label{fig:DIinv}
\end{figure}

\section{Definitions and the Alternating Sign Matrix - Fully Packed Loop - Dyck Island correspondence }
\label{sec:defn}
Let us now clarify the terminology which will be used throughout the paper. We wish to emphasize the harmony of the different representations of fully packed loops. In Section \ref{sec:inv}, we plan to use several representations at once, particularly in the proof of the main theorem. 
\begin{defn}
	A \emph{fully packed loop} of size $n$ is a connected graph arranged in an $n\times n$ grid such that there are $n^2$ internal vertices of degree $4$, and $4n$ external vertices of degree 1. Edges of the graph are colored either light or dark such that all internal vertices are incident to two light edges and two dark edges--this is the six-vertex condition (see Figure \ref{fig:6v}). Furthermore, edges incident to the vertices of degree 1 alternate in color in the manner seen in Figure \ref{fig:phi}. These are the domain wall boundary conditions. The set of all fully packed loops of size $n$ will be denoted $\FPL{n}$.
\end{defn}
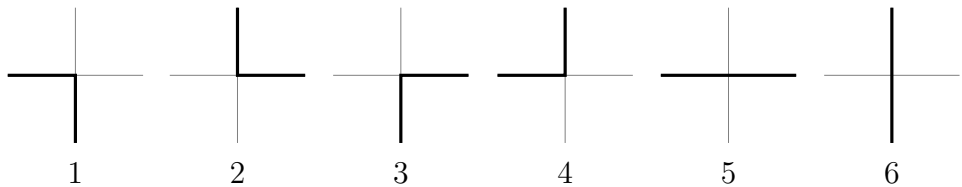
\begin{figure}[h]
	\begin{center}
		\begin{tikzpicture}
			\draw[step=1cm,gray,very thin] (-.9,-.9) grid (.9,.9);
			\draw[very thick] (-.9,0) -- ++(.9,0) -- ++(0,-.9);
			\node at (0,-1.3) {$1$};
		\end{tikzpicture}
		$\,$
		\begin{tikzpicture}
			\draw[step=1cm,gray,very thin] (-.9,-.9) grid (.9,.9);
			\draw[very thick] (.9,0) -- ++(-.9,0) -- ++(0,.9);
			\node at (0,-1.3) {$2$};
		\end{tikzpicture}
		$\,$
		\begin{tikzpicture}
			\draw[step=1cm,gray,very thin] (-.9,-.9) grid (.9,.9);
			\draw[very thick] (.9,0) -- ++(-.9,0) -- ++(0,-.9);
			\node at (0,-1.3) {$3$};
		\end{tikzpicture}
		$\,$
		\begin{tikzpicture}
			\draw[step=1cm,gray,very thin] (-.9,-.9) grid (.9,.9);
			\draw[very thick] (-.9,0) -- ++(.9,0) -- ++(0,.9);
			\node at (0,-1.3) {$4$};
		\end{tikzpicture}
		$\,$
		\begin{tikzpicture}
			\draw[step=1cm,gray,very thin] (-.9,-.9) grid (.9,.9);
			\draw[very thick] (-.9,0) -- ++(1.8,0);
			\node at (0,-1.3) {$5$};
		\end{tikzpicture}
		$\,$
		\begin{tikzpicture}
			\draw[step=1cm,gray,very thin] (-.9,-.9) grid (.9,.9);
			\draw[very thick] (0,-.9) -- ++(0,1.8);
			\node at (0,-1.3) {$6$};
		\end{tikzpicture}
	\end{center}
	\caption{The six-vertex condition}
	\label{fig:6v}
\end{figure}
\begin{defn}
	An \emph{alternating sign matrix} of size $n$ is an $n\times n$ matrix with entries $0$, $1$, or $-1$ such that each row sum is equal to 1, each column sum is equal to 1, and the non-zero entries alternate in sign along both rows and columns.
\end{defn}
Figure \ref{fig:asm} gives two example alternating sign matrices.

\begin{defn}
	Given a fixed alternating sign matrix, $A$, a \emph{diagonal one} of $A$ is an entry along the diagonal which takes the value $1$.
\end{defn}

It is well known (see \cite{ProppManyFaces} for a review) that there exists a bijection between fully packed loops of size $n$ and alternating sign matrices of size $n$. Vertices of type $1-4$ correspond to 0, and vertices of type $5$ and $6$ correspond to 1 and $-1$ subject to the alternating sign condition. 
\begin{figure}[h]
	\begin{center}
		\begin{equation*}
			\begin{pmatrix}
				1 & 0 & 0 & 0 & 0 \\
				0 & 0 & 1 & 0 & 0  \\
				0 & 1 & -1 & 0 & 1  \\
				0 & 0 & 1 & 0 & 0  \\
				0 & 0 & 0 & 1 & 0 
			\end{pmatrix}
			\begin{pmatrix}
				0 & 0 & 0 & 0 & 1 \\
				0 & 0 & 1 & 0 & 0  \\
				0 & 1 & 0 & 0 & 0  \\
				1 & -1 & 0 & 0 & 0  \\
				0 & 1 & 0 & 0 & 0  
			\end{pmatrix}
		\end{equation*}
	\end{center}
	\caption{Two example alternating sign matrices}
	\label{fig:asm}
\end{figure}
\begin{defn}
	Let $\phi_0$ be the fully packed loop corresponding to the identity matrix. 
	Likewise, let $\phi_1$ be the fully packed loop corresponding to the skew-identity matrix. 
\end{defn}
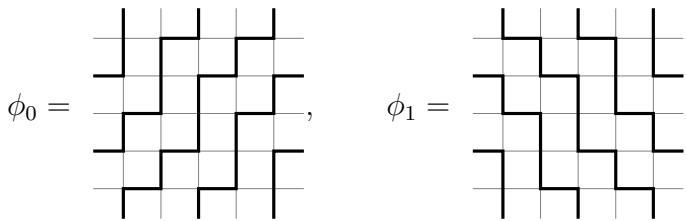
\begin{figure}[h]
	\begin{center}
		$\phi_0=$
		\raisebox{-1.4cm}{
		\begin{tikzpicture}[fpl/.style={very thick}]
			\draw[step=.5cm,gray,very thin] (-1.4,-1.4) grid (1.4,1.4);
			\draw[fpl] (-1,1.4) -- (-1,0.5) -- (-1.4,0.5);
			\draw[fpl] (0,1.4) -- ++(0,-.4) -- ++(-.5,0) -- ++(0,-1) -- ++(-.5,0) -- ++(0,-.5) -- ++(-.4,0);
			\draw[fpl] (1,1.4) -- ++(0,-.4) -- ++(-.5,0)-- ++(0,-.5) -- ++(-.5,0) -- ++(0,-1) -- ++(-.5,0) -- ++(0,-.5) -- ++(-.5,0) -- ++(0,-.4);
			\draw[fpl] (1.4,.5) -- ++(-.4,0) -- ++(0,-.5) -- ++(-.5,0) -- ++(0,-1) -- ++(-.5,0) -- ++(0,-.4);
			\draw[fpl] (1,-1.4) -- (1,-.5) -- (1.4,-.5);
		\end{tikzpicture}},
		$\qquad \phi_1=$
		\raisebox{-1.4cm}{
		\begin{tikzpicture}[fpl/.style={very thick}]
			\draw[step=.5cm,gray,very thin] (-1.4,-1.4) grid (1.4,1.4);
			\draw[fpl] (-1.4,-.5) -- ++(0.4,0) -- ++(0,-.9);
			\draw[fpl] (-1.4,.5) -- ++(.4,0) -- ++(0,-.5) -- ++(.5,0) -- ++(0,-1) -- ++(.5,0) -- ++(0,-.4);
			\draw[fpl] (-1,1.4) -- ++(0,-.4) -- ++(.5,0)-- ++(0,-.5) -- ++(.5,0) -- ++(0,-1) -- ++(.5,0) -- ++(0,-.5) -- ++(.5,0) -- ++(0,-.4);
			\draw[fpl] (0,1.4)-- ++(0,-.4) -- ++(.5,0) -- ++(0,-1) -- ++(.5,0) -- ++(0,-.5) -- ++(.4,0);
			\draw[fpl] (1,1.4) -- ++(0,-.9) -- ++(.4,0);
		\end{tikzpicture}}
	\end{center}
	\caption{The fully packed loops $\phi_0$ and $\phi_1$.}
	\label{fig:phi}
\end{figure}
\begin{defn}
	An \emph{alternating path} is a collection of edges of a fully packed loop which form lattice path loops and for which the edge color alternates. An \emph{alternating loop} is a single loop which has alternating edge colors. Thus, an alternating path is a union of alternating loops. Figure \ref{fig:alt} gives examples of alternating paths in a $5\times5$ fully packed loop corresponding to the alternating sign matrices in Figure \ref{fig:asm}.
\end{defn}
\begin{defn}
	Let $p = \cup_i \gamma_i$ be a union of one or more lattice path loops, $\gamma_i$. Define the flip of $\gamma_i$ to be the map of $\FPL{n}$ to itself which flips the colors of the edges of $\gamma_i$ from light to dark and vice versa if $\gamma_i$ is an alternating loop, and does nothing if $\gamma_i$ is not an alternating loop. We define the \emph{flip of $p$} to be the map from $\FPL{n}$ to itself which flips all $\gamma_i$ which are alternating. A \emph{plaquette flip} is a flip of a loop surrounding a $1\times 1 $ box. 
\end{defn}
\begin{figure}[h]
	\begin{center}
		\begin{tikzpicture}[fpl/.style={very thick},
			alt/.style={line width=1mm,color=red,opacity=.4}]
			\draw[step=.5cm,gray,very thin] (-1.4,-1.4) grid (1.4,1.4);
			\draw[fpl] (-1,1.4) -- (-1,0.5) -- (-1.4,0.5);
			\draw[fpl] (0,1.4) -- ++(0,-.4) -- ++(-.5,0) -- ++(0,-1) -- ++(-.5,0) -- ++(0,-.5) -- ++(-.4,0);
			\draw[fpl] (1,1.4) -- ++(0,-.4) -- ++(-.5,0)-- ++(0,-.5) -- ++(-.5,0) -- ++(0,-1) -- ++(-.5,0) -- ++(0,-.5) -- ++(-.5,0) -- ++(0,-.4);
			\draw[fpl] (1.4,.5) -- ++(-.4,0) -- ++(0,-.5) -- ++(-.5,0) -- ++(0,-1) -- ++(-.5,0) -- ++(0,-.4);
			\draw[fpl] (1,-1.4) -- (1,-.5) -- (1.4,-.5);
			\draw[alt] (0,0) rectangle (-.5,.5);
			\draw[alt] (0,0) -- ++(1,0) -- ++(0,-1) -- ++(-.5,0) -- ++(0,.5) -- ++(-.5,0) -- cycle;
		\end{tikzpicture}
		$\qquad$
		\begin{tikzpicture}[fpl/.style={very thick},
			alt/.style={line width=1mm,color=red,opacity=.4}]
			\draw[step=.5cm,gray,very thin] (-1.4,-1.4) grid (1.4,1.4);
			\draw[fpl] (-1,1.4) -- (-1,0.5) -- (-1.4,0.5);
			\draw[fpl] (0,1.4) -- ++(0,-.4) -- ++(-.5,0) -- ++(0,-1) -- ++(-.5,0) -- ++(0,-.5) -- ++(-.4,0);
			\draw[fpl] (1,1.4) -- ++(0,-.4) -- ++(-.5,0)-- ++(0,-.5) -- ++(-.5,0) -- ++(0,-1) -- ++(-.5,0) -- ++(0,-.5) -- ++(-.5,0) -- ++(0,-.4);
			\draw[fpl] (1.4,.5) -- ++(-.4,0) -- ++(0,-.5) -- ++(-.5,0) -- ++(0,-1) -- ++(-.5,0) -- ++(0,-.4);
			\draw[fpl] (1,-1.4) -- (1,-.5) -- (1.4,-.5);
			\draw[alt] (0,0) rectangle (-.5,.5);
			\draw[alt] (-1,1) -- ++(2,0) -- ++(0,-2) -- ++(-1.5,0) -- ++(0,.5) -- ++(-.5,0) -- cycle;
		\end{tikzpicture}
	\end{center}
	\caption{Alternating paths on $\phi_0$.}
	\label{fig:alt}
\end{figure}
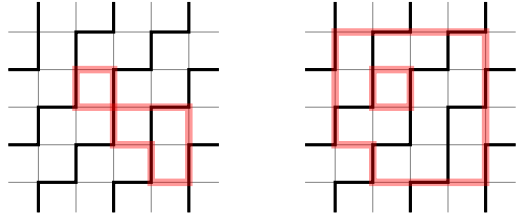

\begin{defn}
	A \emph{Dyck island} of size $n+1$ is an $n\times n$ tableau filled with entries, $\delta_{ij}$ for $1 \leq i,j \leq n$, from $\{0,1,2,3, \ldots\}$ according to the following rules:
	\begin{itemize}
		\item $\delta_{ij} \geq \delta_{i^\prime,j^\prime}$ whenever $i \geq j, i\leq i^\prime$, and $j \geq j^\prime$ 
		\item $\delta_{ij} \geq \delta_{i^\prime, j^\prime}$ whenever $i \leq j, i \geq i^\prime$ and $j \leq j^\prime$
		\item $\delta_{ij} = 0 \textrm{ or } 1 $ if $i,j \in \{1,n\}$
		\item $|\delta_{ij}-\delta_{(i+1)j}| \leq 1$ and  $|\delta_{ij}-\delta_{i(j+1)}| \leq 1$.
	\end{itemize}
	\label{def:di}
\end{defn}
In words, the above definition tells us the following: Fix a box on the diagonal. Entries of boxes above and to the right are weakly decreasing, by increments of at most 1 per step. Likewise, entries of boxes below and to the left are weakly decreasing, by increments of at most 1 per step. 
Furthermore, boxes along the boundary can only take the values 0 or 1. Superimposing a Dyck island over the fully packed loop $\phi_0$ specifies an alternating path along the boundaries of level sets where the value inside a box in a Dyck island indicates the number of alternating paths which contain the box. Figure \ref{fig:exampleDI} shows the two Dyck islands corresponding to the alternating paths of Figure \ref{fig:alt}. 

\begin{figure}[h]
	\begin{center}
		\begin{tikzpicture}
			\draw[step=.5cm,gray,very thin] (-1,-1) grid (1,1);
			\node at (-.75,.75) {$0$};
			\node at (-.25,.75) {$0$};
			\node at (.25,.75) {$0$};
			\node at (.75,.75) {$0$};
			\node at (-.75,.25) {$0$};
			\node at (-.25,.25) {$1$};
			\node at (.25,.25) {$0$};
			\node at (.75,.25) {$0$};
			\node at (-.75,-.25) {$0$};
			\node at (-.25,-.25) {$0$};
			\node at (.25,-.25) {$1$};
			\node at (.75,-.25) {$1$};
			\node at (-.75,-.75) {$0$};
			\node at (-.25,-.75) {$0$};
			\node at (.25,.-.75) {$0$};
			\node at (.75,-.75) {$1$};
			\draw[very thick] (0,0) rectangle (-.5,.5);
			\draw[very thick] (0,0) -- ++(1,0) -- ++(0,-1) -- ++(-.5,0) -- ++(0,.5) -- ++(-.5,0) -- cycle;
		\end{tikzpicture}
		$\qquad$
		\begin{tikzpicture}
			\draw[step=.5cm,gray,very thin] (-1,-1) grid (1,1);
			\node at (-.75,.75) {$1$};
			\node at (-.25,.75) {$1$};
			\node at (.25,.75) {$1$};
			\node at (.75,.75) {$1$};
			\node at (-.75,.25) {$1$};
			\node at (-.25,.25) {$2$};
			\node at (.25,.25) {$1$};
			\node at (.75,.25) {$1$};
			\node at (-.75,-.25) {$1$};
			\node at (-.25,-.25) {$1$};
			\node at (.25,-.25) {$1$};
			\node at (.75,-.25) {$1$};
			\node at (-.75,-.75) {$0$};
			\node at (-.25,-.75) {$1$};
			\node at (.25,.-.75) {$1$};
			\node at (.75,-.75) {$1$};
			\draw[very thick] (0,0) rectangle (-.5,.5);
			\draw[very thick] (-1,1) -- ++(2,0) -- ++(0,-2) -- ++(-1.5,0) -- ++(0,.5) -- ++(-.5,0) -- cycle;
		\end{tikzpicture}
	\end{center}
	\caption{Two example Dyck islands. Boundary paths have been drawn in bold.}
	\label{fig:exampleDI}
\end{figure}

It is possible to construct all Dyck islands inductively from the Dyck island of all zero entries according to the following local update rules:
\begin{itemize}
	\item
	$\raisebox{-1cm}{
	\begin{tikzpicture}[scale=1.3]
		\draw[gray, very thin] (-.75,-.25) rectangle (.75,.25);
		\draw[gray, very thin] (-.25,.75) rectangle (.25,-.75);
		\node at (0,0) {$i$};
		\node at (-.5,0) {$i$};
		\node at (.5,0) {$i$};
		\node at (0,-.5) {$i$};
		\node at (0,.5) {$i$};
	\end{tikzpicture}
	}
	\leftrightarrow
	\raisebox{-1cm}{
	\begin{tikzpicture}[scale=1.3]
		\draw[gray, very thin] (-.75,-.25) rectangle (.75,.25);
		\draw[gray, very thin] (-.25,.75) rectangle (.25,-.75);
		\node at (0,0) {\tiny{$i+1$}};
		\node at (-.5,0) {$i$};
		\node at (.5,0) {$i$};
		\node at (0,-.5) {$i$};
		\node at (0,.5) {$i$};
	\end{tikzpicture}
	}$ when the entry to be updated is on the diagonal.
\smallskip
		\item
	$\raisebox{-1cm}{
	\begin{tikzpicture}[scale=1.3]
		\draw[gray, very thin] (-.75,-.25) rectangle (.75,.25);
		\draw[gray, very thin] (-.25,.75) rectangle (.25,-.75);
		\node at (0,0) {$i$};
		\node at (-.5,0) {\tiny{$i+1$}};
		\node at (.5,0) {$i$};
		\node at (0,-.5) {\tiny{$i+1$}};
		\node at (0,.5) {$i$};
	\end{tikzpicture}
	}
	\leftrightarrow
	\raisebox{-1cm}{
	\begin{tikzpicture}[scale=1.3]
		\draw[gray, very thin] (-.75,-.25) rectangle (.75,.25);
		\draw[gray, very thin] (-.25,.75) rectangle (.25,-.75);
		\node at (0,0) {\tiny{$i+1$}};
		\node at (-.5,0) {\tiny{$i+1$}};
		\node at (.5,0) {$i$};
		\node at (0,-.5) {\tiny{$i+1$}};
		\node at (0,.5) {$i$};
	\end{tikzpicture}
	}$ when the entry to be updated is above the diagonal.
	\smallskip
	\item
	$\raisebox{-1cm}{
	\begin{tikzpicture}[scale=1.3]
		\draw[gray, very thin] (-.75,-.25) rectangle (.75,.25);
		\draw[gray, very thin] (-.25,.75) rectangle (.25,-.75);
		\node at (0,0) {$i$};
		\node at (-.5,0) {$i$};
		\node at (.5,0) {\tiny{$i+1$}};
		\node at (0,-.5) {$i$};
		\node at (0,.5) {\tiny{$i+1$}};
	\end{tikzpicture}
	}
	\leftrightarrow
	\raisebox{-1cm}{
	\begin{tikzpicture}[scale=1.3]
		\draw[gray, very thin] (-.75,-.25) rectangle (.75,.25);
		\draw[gray, very thin] (-.25,.75) rectangle (.25,-.75);
		\node at (0,0) {\tiny{$i+1$}};
		\node at (-.5,0) {$i$};
		\node at (.5,0) {\tiny{$i+1$}};
		\node at (0,-.5) {$i$};
		\node at (0,.5) {\tiny{$i+1$}};
	\end{tikzpicture}
	}$ when the entry to be updated is below the diagonal.
\end{itemize}
For the purpose of making sense of the update rules along the boundary, assume that the Dyck island has an additional first and last row of 0 entries and an additional first and last column of 0 entries. 

\begin{defn}
	Following \cite{ProppManyFaces}, the \emph{height function} representation of an $n\times n$ alternating sign matrix, $A_{i^\prime j^\prime}$, is an $(n+1) \times (n+1)$ matrix
\begin{equation*}
	h_{ij} = i+j - 2 \left(\sum_{i^\prime = 1}^i \sum_{j^\prime = 1}^j A_{i^\prime j^\prime} \right)
\end{equation*}
where $0\leq i,j \leq n$, and the sums are to be zero when $i=0$ or $j=0$, so that $h_{ij}=i+j$ whenever $i=0$ or $j=0$. See Figure \ref{fig:heightFunc} to see examples of height functions which correspond to the Dyck islands in Figure \ref{fig:exampleDI}.
\end{defn}
\begin{figure}[h]
	\begin{center}
		\begin{equation*}
			\begin{pmatrix}
				0 & 1 & 2 & 3 & 4 & 5 \\
				1 & 0 & 1 & 2 & 3 & 4 \\
				2 & 1 & 2 & 1 & 2 & 3 \\
				3 & 2 & 1 & 2 & 3 & 2 \\
				4 & 3 & 2 & 1 & 2 & 1 \\
				5 & 4 & 3 & 2 & 1 & 0 
			\end{pmatrix}
			\begin{pmatrix}
				0 & 1 & 2 & 3 & 4 & 5 \\
				1 & 2 & 3 & 4 & 5 & 4 \\
				2 & 3 & 4 & 1 & 4 & 3 \\
				3 & 4 & 3 & 2 & 3 & 2 \\
				4 & 3 & 4 & 3 & 2 & 1 \\
				5 & 4 & 3 & 2 & 1 & 0 
			\end{pmatrix}
		\end{equation*}
	\end{center}
	\caption{Two example height functions}
	\label{fig:heightFunc}
\end{figure}

In \cite{LS}, Lascoux and Schutzenberger showed that monotone triangles (yet another object in bijection with alternating sign matrices--see \cite{ProppManyFaces} for examples) satisfy an interesting lattice structure, which by the above bijection, carries through to the height function representation. The infimum and supremum of the entire set is the identity and skew-identity, respectively. In the height function representation, there is a particularly easy interpretation of the lattice structure: two height functions, $h^{(1)}$ and $h^{(2)}$, satisfy $h^{(1)} \leq h^{(2)}$ if and only if  $h^{(1)}_{ij} \leq h^{(2)}_{ij}$ for $0\leq i,j \leq n$. What is more, entries of the height function differ by even numbers, and the border entries ($i,j \in {0,n}$) remain constant. This motivates the following formula which gives a bijection between Dyck islands and height functions:

Let $h^{(0)}$ correspond to the minimal $(n+1) \times (n+1)$ height function (which corresponds to the identity matrix in the alternating sign matrix picture) and let $h$ be any given $(n+1) \times (n+1)$ height function. Then we get a corresponding Dyck island $\delta$ via
\begin{equation*}
	\delta_{ij} = \frac 1 2 (h_{ij} - h^{(0)}_{ij}) \textrm{ where } 1 \leq i, j \leq n-1 . 
\end{equation*}
Because the entries with $i \in \{ 0 , n+1\}$ or $j \in \{0, n+1\}$ remain constant in the height function representation, it is clear that the above map is a bejection. We get the following proposition.
\begin{prop}
	Dyck islands are in bijection with fully packed loops and alternating sign matrices.
\end{prop}

We now introduce terminology which is useful for describing any specified Dyck island.
\begin{defn}
	A \emph{Dyck word of semilength $n$} is a word of length $2n$ from the alphabet $\{u,d\}$ such that the number of `$u$' and `$d$' are equal, and in all of the subwords consisting of the first consecutive $i$ letters, the number of `$u$' always exceeds or is equal to the number of `$d$' for each $i$ in the range $1\leq i \leq 2n$. 
\end{defn}
\begin{defn}
	A \emph{Dyck path of semilength $n$} is a lattice path in a finite size square lattice constructed from a Dyck word of semilength $n$ in which the path begins at some point along the diagonal and the letters $\{u,d\}$ are interpreted as $\{$right, down$\}$ for paths above the diagonal or as $\{$down, right$\}$ for paths below the diagonal. By construction, a Dyck path begins and ends on the diagonal, and never crosses the diagonal. In other words, we fix whether our path is above the diagonal or below the diagonal and interpret $u$ to be a move away from the diagonal and $d$ to be a move toward the diagonal.
\end{defn}

\begin{defn}
	Given any Dyck island of size $n$, we notice that entries take values in the set $\{0, 1, \ldots , \lceil\frac{n}{2} \rceil\}$. Consider the union of all boxes labelled $i$ such that $1 \leq i \leq \lceil \frac{n}{2} \rceil$. By construction, these regions are bounded by boxes labelled $i-1$ or $i+1$. The union of all edges between such regions for all $i \in \{0,1,\ldots, \lceil \frac{n}{2} \rceil \}$ are lattice path loops, which we call \emph{boundaries} or \emph{boundary paths}. By specifying all boundaries, one can retrieve the entries of the Dyck island by inserting in each entry the number of boundaries which contain the box (in the interior of the boundary) in consideration. See Figure \ref{fig:exampleDI} for examples.
\end{defn}

Boundaries are a union of lattice path loops which are allowed to touch along vertices of the underlying graph. We fix the convention that we decompose the boundaries of a Dyck island into loops which can be described by two Dyck paths of equivalent semilength: one which forms the northeastern side of the boundary, and the other which forms the southwestern side.  One can check that boundaries decompose into different Dyck paths under the local update rules for Dyck islands, but the existence of the Dyck path representation is preserved. 

\begin{defn}
	For any given lattice path loop, $\gamma$, which is assumed to form part of the boundary of a given Dyck island, we denote the northeast boundary of $\gamma$ with $\neb(\gamma)$ and we denote the southwestern boundary of $\gamma$ with $\swb(\gamma)$. Both $\neb(\gamma)$ and $\swb(\gamma)$ are Dyck paths. See Figure \ref{fig:nebswbdemo} for an example.
\end{defn}
\begin{defn}
	Let $\pi_1$ and $\pi_2$ be two distinct Dyck paths which are part of the boundary of a given Dyck island (possibly from the same loop). An \emph{osculation} is a point of the lattice which $\pi_1$ and $\pi_2$ share in common. 
\end{defn}

In order to eliminate ambiguities in decomposing the boundaries, we fix the convention that loops with osculations along the diagonal cannot be broken up into smaller loops. For example, in Figure \ref{fig:DIinv}, the first Dyck island is described by two loops and the remaining three are described by one loop. Figure \ref{fig:nebswbdemo} gives another example.

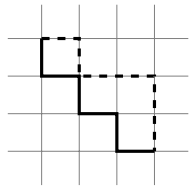
\begin{figure}[h]
	\begin{center}
		\begin{tikzpicture}[scale=.5]
			\draw[very thin, gray] (-.9,-.9) grid (3.9,3.9);
			\draw[very thick, dashed] (0,3) -- (1,3) -- (1,2) -- (3,2) -- (3,0);
			\draw[very thick] (0,3) -- (0,2) -- (1,2) -- (1,1) -- (2,1) -- (2,0) -- (3,0);
		\end{tikzpicture}
	\end{center}
	\caption{An example lattice path loop of semilength $3$ in the boundary of a Dyck island. The northeast boundary of the given lattice loop is labelled with a dashed path. The southwest bondary is labelled with a solid path. Note that by our convention, we will always describe these boundaries as a single loop rather than two loops.}
	\label{fig:nebswbdemo}
\end{figure}

Observe that distinct Dyck paths forming the boundary of a Dyck island may not share an edge because this would violate the fourth condition of Definition \ref{def:di}. Hence the Dyck paths which form the boundaries of Dyck islands only touch at isolated points.

We now wish to demonstrate how the notion of a Dyck island is related to the set of alternating paths on the fully packed loop $\phi_0$ of arbitrary size $n$. We will find that by reinterpreting alternating paths as the union of lattice path loops in the square lattice of size $n$, we recover the boundary of a Dyck island. Though Dyck islands are perhaps most simply defined via the connection to height functions, their discovery arose through the study of alternating paths.

\begin{lemma}
	Any alternating loop can be constructed by some sequence of plaquette flips. 
	\label{lem:plaquetteConstruction}
\end{lemma}
\begin{proof}
	The following proof works for any given fully packed loop. With a fixed fully packed loop and alternating path in mind, it is clear that if we can apply a single plaquette flip to all boxes in the interior, then each interior edge is flipped twice, while each exterior edge is only flipped once. Thus, such a sequence of plaquette flips implements an alternating path. 
	
	Let us call a box \emph{accessible} if it can be flipped by a plaquette flip eventually, after some sequence of plaquette flips in the interior. We wish to demonstrate that all plaquette flips in the interior of an alternating path are accessible. The proof is by induction on the number of boxes in the interior. The case of one box is obvious, since this is simply a plaquette flip to begin with. The inductive step is demonstrated by cutting up the interior of the alternating path into two parts, where we cut along some alternating path. Then one of the two regions is bounded by an alternating path and the other region can be shown to be bounded by an alternating path upon a color flip operation applied to the cutting path. The number of boxes that each of these smaller alternating path bounds is smaller than $n$, therefore by the inductive hypothesis, all of the boxes within are accessible. 
	Lastly, to demonstrate that such an alternating cutting path exists, we make the observation that every edge is part of some alternating path, by the six-vertex condition. Then, if we pick any edge that is both incident to a vertex on the alternating path and in the interior of the alternating path and use this edge to find a new alternating cutting path, we see that the cutting path must be incident to our original alternating path in at least 2 points. 
\end{proof}
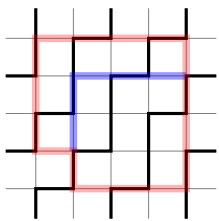
\begin{figure}[h]
	\begin{center}
				\begin{tikzpicture}[fpl/.style={very thick},
			alt/.style={line width=1mm,color=red,opacity=.3}]
			\draw[step=.5cm,gray,very thin] (-1.4,-1.4) grid (1.4,1.4);
			\draw[fpl] (-1,1.4) -- (-1,0.5) -- (-1.4,0.5);
			\draw[fpl] (0,1.4) -- ++(0,-.4) -- ++(-.5,0) -- ++(0,-1) -- ++(-.5,0) -- ++(0,-.5) -- ++(-.4,0);
			\draw[fpl] (1,1.4) -- ++(0,-.4) -- ++(-.5,0)-- ++(0,-.5) -- ++(-.5,0) -- ++(0,-1) -- ++(-.5,0) -- ++(0,-.5) -- ++(-.5,0) -- ++(0,-.4);
			\draw[fpl] (1.4,.5) -- ++(-.4,0) -- ++(0,-.5) -- ++(-.5,0) -- ++(0,-1) -- ++(-.5,0) -- ++(0,-.4);
			\draw[fpl] (1,-1.4) -- (1,-.5) -- (1.4,-.5);
			\draw[line width=1mm,color=blue,opacity=.4] (-.5,-.5) -- ++(0,1) -- ++(1.5,0);
			\draw[alt] (-1,1) -- ++(2,0) -- ++(0,-2) -- ++(-1.5,0) -- ++(0,.5) -- ++(-.5,0) -- cycle;
		\end{tikzpicture}
	\end{center}
	\caption{The path colored blue illustrates one possible cutting path.}
	\label{fig:cutting}
\end{figure}

Plaquette flips in the fully packed loop picture correspond to the local update rules in the Dyck island picture. This is the content of the next proposition.

\begin{prop}
	Fix a box $\alpha$ in the $n\times n$ square lattice. Let $f_\alpha$ be the operator acting on $\FPL{n}$ which implements a plaquette flip on the box $\alpha$ if it is surrounded by an alternating path and which does nothing otherwise. Let $u_\alpha$ be the operator on $(n-1) \times (n-1)$ Dyck islands which implements a local update at the box $\alpha$ (adds $\pm1$ to the box $\alpha$) if it is permissible, and does nothing otherwise. Then there exists a bijection $M$ from $\FPL{n}$ to the $(n-1)\times(n-1)$ Dyck islands such that 
	\begin{equation*}
		M\circ f_\alpha = u_\alpha \circ M.
	\end{equation*}
Furthermore, $M$ is the map which forgets information about color and reinterprets alternating paths of $\phi_0$ as boundary paths in Dyck islands.
\label{prop:bijectionM}
\end{prop}
\begin{proof}
We will construct the bijection $M$. 

First, we establish a bijective correspondence between alternating paths of $\phi_0$ and the set of fully packed loops. Observe that a simple consequence of the six-vertex condition  and the boundary condition is that for any fixed fully packed loop $\phi$, all alternating paths close up into a union of  loops. Since the six vertex condition guarantees that if any two fully packed loops differ at a vertex, they must differ along two (one black and one white) or all four edges. It follows then that the set of all edges which differ between $\phi$ and $\phi_0$ (the fully packed loop corresponding to the identity matrix) is an alternating path which we denote $\tilde{\gamma}=\gamma_1 \cup \ldots \cup \gamma_\ell$.  Thus, we see that the flip of $\tilde{\gamma}$ maps $\phi$ to $\phi_0$ and vice versa. The upshot of this is that it is possible to obtain every fully packed loop as the flip of some alternating path of $\phi_0$. 

Next, we observe that an arbitrary alternating path loop, $\gamma$, in $\phi_0$ traverses at least two points along the diagonal, that $\gamma$ can be decomposed into a portion above the diagonal ($\neb(\gamma)$) and a portion below the diagonal ($\swb(\gamma)$), and that $\neb(\gamma)$ and $\swb(\gamma)$ are Dyck paths. This follows because all off-diagonal elements of $\phi_0$ are only of type $3$ or $4$ (see Figures \ref{fig:6v} and \ref{fig:phi}). Because alternating paths of $\phi_0$ are a union of alternating path loops $\gamma_1, \ldots, \gamma_\ell$, we see that by forgetting the coloring of a given alternating path we can reinterpret it as the boundary path of a Dyck island. 

We define our bijection $M$ from $\FPL{n}$ to $(n-1)\times(n-1)$ Dyck islands to be the map obtained from the following process: 
\begin{center}
	\begin{enumerate}
		\item Given $\phi$, find the corresponding alternating path of $\phi_0$. Call it $\tilde{\gamma}$. 
		\item Interpret $\tilde{\gamma}$ as a boundary path of a Dyck island, $\delta$ , and fill in the entries $\delta_{ij}$ according to the number of boundary path loops which contain the box $(i,j)$. 
	\end{enumerate}
\end{center}
	This process may be completed in reverse, so it follows that $M$ is a bijection.

	Lastly, in order to establish that 
	\begin{equation*}
		M\circ f_\alpha = u_\alpha \circ M.
	\end{equation*}
we simply observe that the local update rules for Dyck islands correspond to an application of $f_\alpha$ to some accessible box $\alpha$. This is because an accessible box $\alpha$ corresponds to a box of one of the three following types:
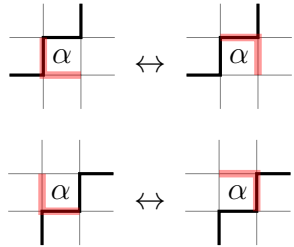
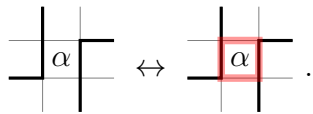
\begin{equation*}
	\raisebox{-.5cm}{
\begin{tikzpicture}[scale=.5]
	\draw[very thin, gray] (-.9,-.9) grid (1.9,1.9);
	\draw[very thick] (1,1.9) -- (1,1) -- (0,1) -- (0,0) -- (-.9,0);
	\draw[line width=1mm,color=red,opacity=.4] (0,1) -- (0,0) -- (1,0);
	\node at (.5,.5) {$\tiny{\alpha}$};
\end{tikzpicture}
} \leftrightarrow
	\raisebox{-.5cm}{
\begin{tikzpicture}[scale=.5]
	\draw[very thin, gray] (-.9,-.9) grid (1.9,1.9);
	\draw[very thick] (1,1.9) -- (1,1) -- (0,1) -- (0,0) -- (-.9,0);
	\draw[line width=1mm,color=red,opacity=.4] (0,1) -- (1,1) -- (1,0);
	\node at (.5,.5) {$\tiny{\alpha}$};
\end{tikzpicture}
}
\end{equation*}
\begin{equation*}
	\raisebox{-.5cm}{
\begin{tikzpicture}[scale=.5]
	\draw[very thin, gray] (-.9,-.9) grid (1.9,1.9);
	\draw[very thick] (1.9,1) -- (1,1) -- (1,0) -- (0,0) -- (0,-.9);
	\draw[line width=1mm,color=red,opacity=.4] (0,1) -- (0,0) -- (1,0);
	\node at (.5,.5) {$\tiny{\alpha}$};
\end{tikzpicture}
} \leftrightarrow
	\raisebox{-.5cm}{
\begin{tikzpicture}[scale=.5]
	\draw[very thin, gray] (-.9,-.9) grid (1.9,1.9);
	\draw[very thick] (1.9,1) -- (1,1) -- (1,0) -- (0,0) -- (0,-.9);
	\draw[line width=1mm,color=red,opacity=.4] (0,1) -- (1,1) -- (1,0);
	\node at (.5,.5) {$\tiny{\alpha}$};
\end{tikzpicture}
}
\end{equation*}
\begin{equation*}
	\raisebox{-.5cm}{
\begin{tikzpicture}[scale=.5]
	\draw[very thin, gray] (-.9,-.9) grid (1.9,1.9);
	\draw[very thick] (-.9,0) -- (0,0) -- (0,1.9);
	\draw[very thick] (1,-.9) -- (1,1) -- (1.9,1);
	\node at (.5,.5) {$\tiny{\alpha}$};
\end{tikzpicture}
} \leftrightarrow
	\raisebox{-.5cm}{
\begin{tikzpicture}[scale=.5]
	\draw[very thin, gray] (-.9,-.9) grid (1.9,1.9);
	\draw[very thick] (-.9,0) -- (0,0) -- (0,1.9);
	\draw[very thick] (1,-.9) -- (1,1) -- (1.9,1);
	\draw[line width=1mm,color=red,opacity=.4] (0,0) rectangle (1,1);
	\node at (.5,.5) {$\tiny{\alpha}$};
\end{tikzpicture}
}.
\end{equation*}
\end{proof}
Because it is more convenient for our analysis of the inversion number, in the remainder of the paper we use the Dyck islands picture. The following proposition explicitly establishes the correspondence between alternating sign matrices and Dyck islands.

\begin{prop}
	Let a Dyck island be described by a boundary path $\tilde{\gamma}=\gamma_1\cup \ldots \cup \gamma_\ell$. Let $w$ be the Dyck word corresponding to $\neb(\gamma_i)$ (respectively $\swb(\gamma_i)$). Let $v$ be a vertex of the underlying lattice corresponding to a consecutive subword $ud$ or $du$ of $\neb(\gamma_i)$ ($\swb(\gamma_i)$). Generically, $v$ corresponds to a $1$ if the subword is $ud$ and $v$ corresponds to a $-1$ if the subword is $du$. All other vertices correspond to $0$. The exceptional cases deal with vertices where distinct loops touch, and where loops traverse a vertex along the diagonal. They are:
	\begin{enumerate}
		\item If $v$ is a vertex along the diagonal which is traversed by either $\neb(\gamma_i)$ or $\swb(\gamma_i)$ but not both, then $v$ corresponds to a $0$ in the alternating sign matrix picture.
		\item If $v$ is a vertex along the diagonal which is traversed by both $\neb(\gamma_i)$ and $\swb(\gamma_i)$, then $v$ corresponds to a $-1$ in the alternating sign matrix picture. 
		\item If $v$ is a vertex not along the diagonal which is traversed by both $\neb(\gamma_i)$ and $\neb(\gamma_j)$ (or is traversed by both $\swb(\gamma_i)$ and $\swb(\gamma_j)$), then $v$ corresponds to a $0$ in the alternating sign matrix picture.
	\end{enumerate}
\end{prop}
\begin{proof}
	The proof proceeds by applying the map $M^{-1}$ where $M$ is the bijection from Proposition \ref{prop:bijectionM} and carefully examining the resulting picture. In all of the diagrams below, red lines indicate the location of the action of a flip of an alternating loop. 

	The generic picture is established by observing four scenarios:
	\begin{equation*}
		\raisebox{-.5cm}
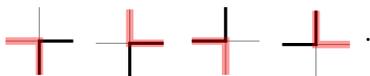
{
			\begin{tikzpicture}[scale=.5]
				\draw[very thin,gray] (-.9,-.9) grid (.9,.9);
				\draw[very thick] (0,-.9) -- (0,0) -- (.9,0);
				\draw[line width=1mm,color=red,opacity=.4] (-.9,0) -- (0,0) -- (0,-.9);
			\end{tikzpicture}
		}
		\raisebox{-.5cm}{
			\begin{tikzpicture}[scale=.5]
				\draw[very thin,gray] (-.9,-.9) grid (.9,.9);
				\draw[very thick] (0,-.9) -- (0,0) -- (.9,0);
				\draw[line width=1mm,color=red,opacity=.4] (.9,0) -- (0,0) -- (0,.9);
			\end{tikzpicture}
		}
		\raisebox{-.5cm}{
			\begin{tikzpicture}[scale=.5]
				\draw[very thin,gray] (-.9,-.9) grid (.9,.9);
				\draw[very thick] (-.9,0) -- (0,0) -- (0,.9);
				\draw[line width=1mm,color=red,opacity=.4] (-.9,0) -- (0,0) -- (0,-.9);
			\end{tikzpicture}
		}
		\raisebox{-.5cm}{
			\begin{tikzpicture}[scale=.5]
				\draw[very thin,gray] (-.9,-.9) grid (.9,.9);
				\draw[very thick] (-.9,0) -- (0,0) -- (0,.9);
				\draw[line width=1mm,color=red,opacity=.4] (.9,0) -- (0,0) -- (0,.9);
			\end{tikzpicture}
		}.
	\end{equation*}
	A flip of $\tilde{\gamma}$ acting on $\phi_0$ transforms these cases into vertices of type $5$ or $6$ (see Figure \ref{fig:6v}). It follows that such vertices correspond to non-zero entries in the alternating sign matrix picture.
	
	To see that the alternating condition is satisfied in the generic case, suppose that we are in a situation with no exceptions (that is, all loops are disjoint and do not have northeast and southwest boundaries intersecting at any point on the diagonal). We first demonstrate that the alternating condition is satisfied in the case that there is only one loop and then show that it is also satisfied for nested loops. The generic case follows.
	
	Fix a loop $\gamma_i$ and consider the Dyck island described by this single loop. Fix a column for observation such that $\gamma_i$ intersects the column in at least one vertex of type $ud$. If the column intersects the point of $\gamma_i$ furthest to the upper left or if it intersects the point of $\gamma_i$ furthest to the lower right, then the column contains only one vertex of type $ud$ and no vertices of type $du$. Otherwise, $\neb(\gamma_i)$ must intersect the column in one vertex of type $ud$ above one vertex of type $du$, or else it intersects along neither. Likewise, still considering the same column, $\swb(\gamma_i)$ must intersect the column in one vertex of type $du$ above one vertex of type $ud$, or else it intersects along neither. After accounting for a diagonal one, we see that the alternating condition is satisfied. See Figure \ref{fig:column}.
	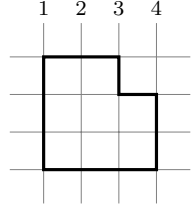
\begin{figure}[h]
		\begin{center}
			\begin{tikzpicture}[scale=.5]
			\draw[very thin, gray] (-.9,-.9) grid (3.9,3.9);
			\draw[very thick] (0,3) -- (2,3) -- (2,2) -- (3,2) -- (3,0) -- (0,0) -- cycle;	
			\node at (0,4.3) {\tiny{1}};
			\node at (1,4.3) {\tiny{2}};
			\node at (2,4.3) {\tiny{3}};
			\node at (3,4.3) {\tiny{4}};
			\end{tikzpicture}
		\end{center}
		\caption{Columns 2 and 3 illustrate how vertices of type $ud$ and $iu$ come in pairs. Columns 1 and 4 illustrates what happens in extremal situations.}
		\label{fig:column}
	\end{figure}

	Suppose we consider a new Dyck island with two loops $\gamma_i$ and $\gamma_j$ such that $\gamma_i$ is contained in the interior of $\gamma_j$. If we pick a column which intersects both $\gamma_i$ and $\gamma_j$, then intersections will be ordered in the following way (reading from top to bottom): First intersections of $\neb(\gamma_j)$, then intersections of $\neb(\gamma_i)$, then possibly a diagonal one, then intersections of $\swb(\gamma_i)$ and lastly, intersections of $\swb(\gamma_j)$. It is easy to check that the alternating condition is still satisfied. Intersections along rows are checked in exactly the same way.
	
	Exception (1) corresponds to one of the following two pictures:
	\begin{equation*}
		\raisebox{-.5cm}
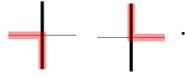
{
			\begin{tikzpicture}[scale=.5]
				\draw[very thin,gray] (-.9,-.9) grid (.9,.9);
				\draw[very thick] (0,-.9) -- (0,.9);
				\draw[line width=1mm,color=red,opacity=.4] (-.9,0) -- (0,0) -- (0,-.9);
			\end{tikzpicture}
		}
		\raisebox{-.5cm}{
			\begin{tikzpicture}[scale=.5]
				\draw[very thin,gray] (-.9,-.9) grid (.9,.9);
				\draw[very thick] (0,-.9) -- (0,.9);
				\draw[line width=1mm,color=red,opacity=.4] (.9,0) -- (0,0) -- (0,.9);
			\end{tikzpicture}
		}.
	\end{equation*}

	Likewise, exception (2) corresponds to
	\begin{equation*}
		\raisebox{-.5cm}
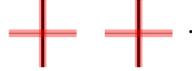
{
			\begin{tikzpicture}[scale=.5]
				\draw[very thin,gray] (-.9,-.9) grid (.9,.9);
				\draw[very thick] (0,-.9) -- (0,.9);
				\draw[line width=1mm,color=red,opacity=.4] (-.9,0) -- (0,0) -- (0,-.9);
				\draw[line width=1mm,color=red,opacity=.4] (.9,0) -- (0,0) -- (0,.9);
			\end{tikzpicture}
		}
		\raisebox{-.5cm}{
			\begin{tikzpicture}[scale=.5]
				\draw[very thin,gray] (-.9,-.9) grid (.9,.9);
				\draw[very thick] (0,-.9) -- (0,.9);
				\draw[line width=1mm,color=red,opacity=.4] (.9,0) -- (0,0) -- (0,.9);
				\draw[line width=1mm,color=red,opacity=.4] (-.9,0) -- (0,0) -- (0,-.9);
			\end{tikzpicture}
		}.
	\end{equation*}

	Lastly, exception (3) corresponds to 

	\begin{equation*}
		\raisebox{-.5cm}
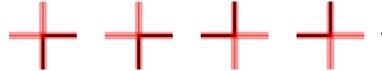
{
			\begin{tikzpicture}[scale=.5]
				\draw[very thin,gray] (-.9,-.9) grid (.9,.9);
				\draw[very thick] (0,-.9) -- (0,0) -- (.9,0);
				\draw[line width=1mm,color=red,opacity=.4] (-.9,0) -- (0,0) -- (0,-.9);
				\draw[line width=1mm,color=red,opacity=.4] (.9,0) -- (0,0) -- (0,.9);
			\end{tikzpicture}
		}
		\raisebox{-.5cm}{
			\begin{tikzpicture}[scale=.5]
				\draw[very thin,gray] (-.9,-.9) grid (.9,.9);
				\draw[very thick] (0,-.9) -- (0,0) -- (.9,0);
				\draw[line width=1mm,color=red,opacity=.4] (.9,0) -- (0,0) -- (0,.9);
				\draw[line width=1mm,color=red,opacity=.4] (-.9,0) -- (0,0) -- (0,-.9);
			\end{tikzpicture}
		}
		\raisebox{-.5cm}{
			\begin{tikzpicture}[scale=.5]
				\draw[very thin,gray] (-.9,-.9) grid (.9,.9);
				\draw[very thick] (-.9,0) -- (0,0) -- (0,.9);
				\draw[line width=1mm,color=red,opacity=.4] (-.9,0) -- (0,0) -- (0,-.9);
				\draw[line width=1mm,color=red,opacity=.4] (.9,0) -- (0,0) -- (0,.9);
			\end{tikzpicture}
		}
		\raisebox{-.5cm}{
			\begin{tikzpicture}[scale=.5]
				\draw[very thin,gray] (-.9,-.9) grid (.9,.9);
				\draw[very thick] (-.9,0) -- (0,0) -- (0,.9);
				\draw[line width=1mm,color=red,opacity=.4] (.9,0) -- (0,0) -- (0,.9);
				\draw[line width=1mm,color=red,opacity=.4] (-.9,0) -- (0,0) -- (0,-.9);
			\end{tikzpicture}
		}.
	\end{equation*}
	Exceptions (1) and (3) are obvious and preserve the alternating condition because they may be interpreted as the merging of adjacent $1$ and $-1$ vertices into a vertex contributing $0$. Exception (2) is established by isolating $\gamma_i$. If we read the vertices in the given column (row) from top (left) to bottom (right), then there is a vertex of type $ud$ before and another vertex of type $ud$ after. $\gamma_i$ does not cross the column (row) in any other locations, so the alternating condition forces the vertex in consideration along the diagonal to be $-1$. 
\end{proof}
\section{Boundary paths of Dyck islands and Inversion number}
\label{sec:inv}
\begin{defn}
	The \emph{inversion number} of an alternating sign matrix $A$, denoted by $\inv(A)$, is 
	\begin{equation*}
		\inv(A) =  \sum_{\substack{1 \leq i, i^\prime, j,j^\prime \leq n\\ i>i^\prime \\j< j^\prime}} A_{ij} A_{i^\prime j^\prime}
	\end{equation*}
\end{defn}
The inversion number is an extension of the standard notion for permutation matrices to all alternating sign matrices. We shall often abuse notation and speak of the inversion number of the associated fully packed loop or Dyck island. 

By only considering non-zero terms, the above sum can be reduced to a sum over pairs of integer tuples $(i,j)$ and $(i^\prime, j^\prime)$ such that $(i^\prime,j^\prime)$ is strictly above and strictly to the right of $(i,j)$ and such that $A_{ij}$ and $A_{i^\prime j^\prime}$ are non-zero. It is easy to see that under a flip across the diagonal, such pairs and the value of $A_{ij}A_{i^\prime j^\prime}$ is preserved. Thus, the following lemma is true.

\begin{lemma}
	Let $A^\prime$ be the reflection of $A$ across the diagonal. Then $\inv(A) = \inv(A^\prime)$.
	\label{lem:invDiag}
\end{lemma}

We now give a way to characterize the inversion number in terms of the loops of a Dyck island. Let $\gamma$ be a loop of a given Dyck island $\delta$. We write $\inv(\gamma)$ to mean the inversion number of a new Dyck island defined by the single boundary loop, $\gamma$. 
\begin{thm}
	If $\delta$ is a Dyck island consisting of $\ell$ loops (possibly nested), $\gamma_1, \ldots , \gamma_\ell$ with a total number of $k$ off-diagonal osculations, then
	\begin{equation*}
		\inv(\delta) = \sum_{n=1}^\ell \inv(\gamma_n) -k
	\end{equation*}
	\label{thm:DIinversions}
\end{thm}
Before proving the theorem, we prove a lemma about evaluating a Dyck island consisting of only one loop $\gamma$.

Recall that the \emph{semilength} of a Dyck path of $2n$ steps ($n$ rises and $n$ falls) is $n$. If the northeast and southwest boundaries of a loop, $\gamma$, are Dyck paths of semilength $n$, then we say also that the semilength of $\gamma$ is $n$. Furthermore, we define an \emph{internal one} to be a diagonal point of the lattice which is strictly contained inside the interior of the loop. In the alternating sign matrix picture, these points correspond to entries along the diagonal with the value 1. 

The analysis of inversion numbers of Dyck islands will be facilitated by the Dyck path structure of the northeast and southwest boundaries of loops. The next lemma characterizes a key property of subpaths of Dyck paths encoded as words in a two letter alphabet. 
\begin{lemma}
	Let $w$ be a word in the alphabet $\{u,d\}$ such that $w$ begins with the letter $u$ and ends with the letter $d$. Then the number of consecutive $ud$ subwords minus the number of consecutive $du$ subwords is exactly 1.
	\label{lem:words}
\end{lemma}
\begin{proof}
	Start with the word $ud$, which obviously evaluates to 1. The insertion of a single $u$ or $d$ in the middle of the word does not change this evaluation. Thus, by checking that the following four insertion scenarios do not change the evaluation, we are done.
	\begin{align*}
		uu & \leftrightarrow uuu \\
		ud & \leftrightarrow uud \\
		du & \leftrightarrow duu \\
		dd & \leftrightarrow dud
	\end{align*}
	The case for insertion of a $d$ is exactly analogous.
\end{proof}
\begin{defn}
	The \emph{contribution zone of vertex $v$} is the rectangular sublattice which lies strictly above and to the right of $v$. Let $CZ(v)$ denote the set of all vertices in the contribution zone of $v$ corresponding to alternating sign matrix entries $1$ or $-1$. The \emph{contribution of the vertex $v$} is the sum over products of pairs of entries given by 
	\begin{equation*}
		\sum_{w \in CZ(v)} A_v A_w
	\end{equation*}
	where $A_v$ denotes the alternating sign matrix entry corresponding to vertex position $v$. By Lemma \ref{lem:invDiag}, we could also assign the contribution zone of $v$ to be below and to the left.
\end{defn}

\begin{lemma}
	Fix an alternating sign matrix of size n corresponding to a Dyck island defined by one path $\gamma$. Let $v$ be a vertex on $\swb(\gamma)$ corresponding to a vertex of type $ud$ or $du$. Then the contribution of all pairs containing the vertex $v$ is $\pm \height(v)$, where the sign corresponds to whether $v$ is a $1$ or $-1$ in the alternating sign matrix picture.  
	\label{lem:heightv}
\end{lemma}
\begin{proof}
	If $v$ is a vertex of type $ud$ or $du$ with height 0 and $v$ is not simultaneously in $\neb(\gamma)$, then we see that $v$ corresponds to neither $1$ nor $-1$ so that its contribution must be 0. When $\neb(\gamma)$ has an vertex corresponding to $du$ at $v$ as well, $v$ corresponds to the alternating sign matrix entry $-1$, but $CZ(v)$ is empty. Thus the result follows when $v$ is height 0. 

	Assume that $\height(v) \geq 1$ and suppose that $v$ corresponds to $1$ (the proof of the case of $-1$ is completely analogous). We wish to demonstrate that the contribution due to vertices of type $ud$ or $du$ in $\neb(\gamma)$ and diagonal ones inside $CZ(v)$ sum to $\height(v)$. Each diagonal one contributes 1 to the sum. The proof will follow once we give a description of the subpath of $\neb(\gamma)$ in $CZ(v)$. 

	First, we claim that the subpath starts at a vertex of type $uu$ or $ud$ and ends in a vertex of type $ud$ or $dd$. If this were not the case, we would be able to add an additional vertex to the subpath. Thus, the word corresponding to the subpath begins with $u$ and ends with $d$. 
	
	Assume that $\neb(\gamma)$ does not touch the diagonal in the contribution zone of $v$. Then there are $\height(v) -1$ diagonal ones and in the word description of the subpath of $\neb(\gamma)$ in $CZ(v)$ each $ud$ corresponds to $1$ and each $du$ corresponds to $-1$. By Lemma \ref{lem:words}, the contribution from the subpath of $\neb(\gamma)$ in $CZ(v)$ is 1, and we conclude that the contribution of all terms in $CZ(v)$ is $\height(v)$. See Figure \ref{fig:contribution1} for an illustration.

	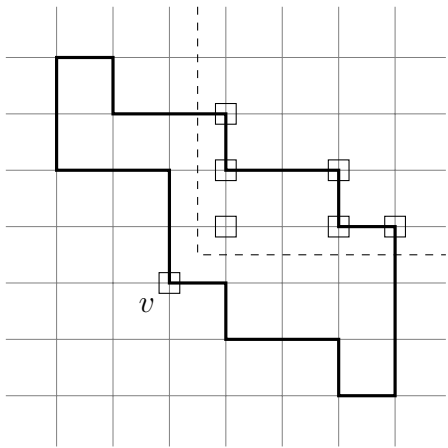
\begin{figure}[h]
		\begin{center}
			\begin{tikzpicture}[scale=.75]
				\draw[very thin, gray] (-.9,-.9) grid (6.9,6.9);
				\draw[very thick] (0,6) -- (1,6) -- (1,5) -- (3,5) -- (3,4) -- (5,4) -- (5,3) -- (6,3) -- (6,0) -- (5,0) -- (5,1) -- (3,1) -- (3,2) -- (2,2) -- (2,4) -- (0,4) -- cycle;
				\node[draw, rectangle] at (2,2) {};
				\node[draw, rectangle] at (3,3) {};
				\node[draw, rectangle] at (3,5) {};
				\node[draw, rectangle] at (3,4) {};
				\node[draw, rectangle] at (5,4) {};
				\node[draw, rectangle] at (5,3) {};
				\node[draw, rectangle] at (6,3) {};
				\node at (1.6,1.6) {$\tiny{v}$};
				\draw[dashed] (2.5,6.9) -- (2.5,2.5) -- (6.9,2.5);
			\end{tikzpicture}
		\end{center}
		\caption{Contributions from the contribution zone of $v$ include a diagonal one. Observe that $\height(v)=2$.}
		\label{fig:contribution1}
	\end{figure}
	Lastly, let us also consider the case when $\neb(\gamma)$ touches the diagonal in the contribution zone of $v$ a total of $k$ times. Then there are $\height(v)-1-k$ diagonal ones. The subpath of $\neb(\gamma)$ in the contribution zone of $v$ must still begin with a vertex of type $uu$ or $ud$ and end with a vertex of type $ud$ or $dd$. In constructing the word corresponding to the subpath, let us mark each vertex which touches the diagonal with $\tilde{d}\tilde{u}$. Such vertices correspond to 0, in the alternating sign matrix picture, but vertices labelled $u\tilde{d}$ or $\tilde{u}d$ still correspond to 1. Thus the contribution from the subpath in $CZ(v)$ is $1+k$ and the total contribution is $\height(v)$. See Figure \ref{fig:contribution2} for an illustration.
	\begin{figure}[h]
		\begin{center}
			\begin{tikzpicture}[scale=.75]
				\draw[very thin, gray] (-.9,-.9) grid (6.9,6.9);
				\draw[very thick] (0,6) -- (1,6) -- (1,5) -- (3,5) -- (3,4) -- (5,4) -- (5,3) -- (6,3) -- (6,0) -- (5,0) -- (5,1) -- (3,1) -- (3,2) -- (2,2) -- (2,4) -- (0,4) -- cycle;
				\node[draw, rectangle] at (0,4) {};
				\node[draw, rectangle] at (1,6) {};
				\node[draw, circle] at (1,5) {};
				\node[draw, rectangle] at (3,5) {};
				\node at (-.4,3.6) {$\tiny{u}$};
				\draw[dashed] (.5,6.9) -- (0.5,4.5) -- (6.9,4.5);
			\end{tikzpicture}
		\end{center}
		\caption{The case when $\neb(\gamma)$ touches the diagonal. Note that the contribution from $u$ is still $\height(u) =2$ since the vertex of type $du$ touching the diagonal (labelled with a circle) contributes 0 instead of $-1$.}
		\label{fig:contribution2}
	\end{figure}
\end{proof}
\begin{lemma}
	Suppose that an $n\times n$ alternating sign matrix, $\gamma$, corresponds to a Dyck island described by a single loop, which we also call $\gamma$, and suppose also that it has $m$ diagonal ones. Then we have
	\begin{equation*}
		\inv(\gamma) = \semilength(\gamma) + m.
	\end{equation*}
	\label{lem:singleLoop}
\end{lemma}
\begin{proof}
We will pair up vertices in the following manner: First, we pair up all diagonal ones with each of the vertices in their respective contribution zones. These vertices will all lie on $\neb(\gamma)$. Second, we pair up all vertices on $\swb(\gamma)$ with each of the vertices in each of their respective contribution zones. These vertices will lie on $\neb(\gamma)$ and also the diagonal ones. Lastly, we remark that all pairs are then accounted for, since there are no vertices above and to the right of $\neb(\gamma)$.

By Lemma \ref{lem:heightv}, each of the vertices on $\swb(\gamma)$ contribute $\height(v)$ for corners of type $ud$ and $-\height(v)$ for corners of type $du$. Diagonal ones contribute 1 to the inversion number sum, since the contribution from a diagonal one comes from pairs which lie on a subpath of $\neb(\gamma)$ to which Lemma \ref{lem:words} applies. 

Thus, it suffices to compute the following alternating sum of the heights:
	\begin{equation*}
		\sum_{x, \textrm{ corners of type } ud} \height(x) - \sum_{y, \textrm{ corners of type } du} \height(y) = \semilength (\gamma)
	\end{equation*}

This sum clearly holds when $\swb(\gamma)$ corresponds to the Dyck word $uu\ldots udd \ldots d$. We show that it holds for any permissible $\swb(\gamma)$ by showing that the sum is invariant under the interchange $ud \leftrightarrow du$ whenever the interchange yields a permissible Dyck word. Locally, there are four cases to check:
\begin{align*}
	\ldots uudd \ldots &\leftrightarrow \ldots udud \ldots \\
	\ldots uudu \ldots &\leftrightarrow \ldots uduu \ldots \\
	\ldots dudd \ldots &\leftrightarrow \ldots ddud \ldots \\
	\ldots dudu \ldots &\leftrightarrow \ldots dduu \ldots \\
\end{align*}
Let us check the first case and remark that the other cases are completely analogous. On the left hand side, locally, we have a single vertex of type $ud$ at height $h$. On the right hand side, this becomes two vertices of type $ud$ at height $h-1$ and one vertex of type $du$ at height $h-2$. The corresponding sum for the right hand side is $2(h-1) -(h-2) = h$. 
\end{proof}
\begin{defn}
	An \emph{off-diagonal osculation} is a vertex $v$ not on the diagonal which lies on $\neb(\gamma_1)$ and $\swb(\gamma_2)$ for two distinct boundary paths $\gamma_1$ and $\gamma_2$ in a Dyck island. See Figure \ref{fig:osculation} for an example.
\end{defn}

We will need the following technical notation in order to complete the proof of Theorem \ref{thm:DIinversions}.

\begin{defn}
	Consider some $n\times n$ alternating sign matrix corresponding to a Dyck island with boundary paths $\gamma_1, \ldots, \gamma_\ell$. We define $\nesting(\gamma_i)$ to be the number of loops in the set $\{\gamma_1,\ldots,\gamma_\ell\}$, excluding $\gamma_i$, which contain $\gamma_i$ in its interior.
	Likewise, if this given alternating sign matrix has diagonal ones located at vertices $p_1, \ldots, p_m$, then $\nesting(p_i)$ denotes the number of loops in the set $\{\gamma_1, \ldots,\gamma_\ell\}$ which contain $p_i$ in its interior. 
\end{defn}
\begin{defn}
	If a Dyck island, $\delta$, is given and is described by the boundary paths $\gamma_1,\ldots, \gamma_\ell$, consider a new Dyck island, denoted by $\delta_i$, which is described by the single path $\gamma_i$ for some $1\leq i \leq \ell$. Let $\diag(\gamma_i)$ be the number of diagonal ones of $\delta_i$ contained inside $\gamma_i$. 
\end{defn}
\begin{defn}
	Let $\gamma$ be a boundary path in some fixed Dyck island. We define $\mathcal{O}^\gamma$ to be the set of vertices on the diagonal which are also vertices of $\gamma$. We define $\mathcal{O}^\gamma_{sw}$ to be the subset of $\mathcal{O}^\gamma$ which restricts to vertices on $\swb(\gamma)$. We note that by our conventions, it is possible for a vertex to be a vertex of $\swb(\gamma)$ and $\neb(\gamma)$ simultaneously. 
\end{defn}
\begin{defn}
	Consider some $n\times n$ alternating sign matrix corresponding to a Dyck island with boundary paths $\gamma_1, \ldots, \gamma_\ell$. The \emph{contribution of $\neb(\gamma_i)$ to the inversion sum}, denoted $\C(\neb(\gamma_i))$, is the sum of the contributions from all of the vertices in $\neb(\gamma_i)$. In exactly the same way, we define the \emph{contribution of $\swb(\gamma_i)$ to the inversion sum} and denote it by $\C(\neb(\gamma_i))$. If a vertex is in both $\neb(\gamma_i)$ and $\swb(\gamma_i)$, then in order to avoid double counting, we establish the convention that it contributes to $\C(\neb(\gamma_i))$ but not to $\C(\swb(\gamma_i))$.
\end{defn}

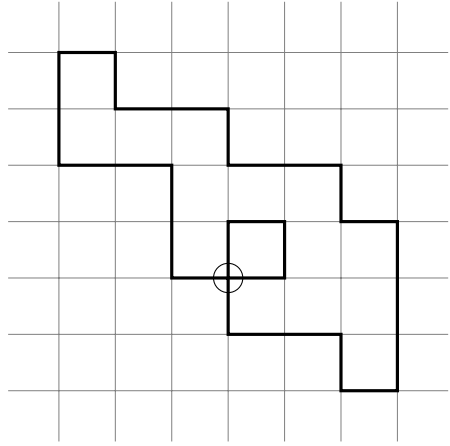
\begin{figure}[h]
	\begin{center}
		\begin{tikzpicture}[scale=.75]
				\draw[very thin, gray] (-.9,-.9) grid (6.9,6.9);
				\draw[very thick] (0,6) -- (1,6) -- (1,5) -- (3,5) -- (3,4) -- (5,4) -- (5,3) -- (6,3) -- (6,0) -- (5,0) -- (5,1) -- (3,1) -- (3,2) -- (2,2) -- (2,4) -- (0,4) -- cycle;
				\draw[very thick] (3,2) rectangle (4,3);
				\node[draw,circle] at (3,2) {};
		\end{tikzpicture}
	\end{center}
	\caption{An example of a Dyck island with two nested boundary paths. The only off-diagonal osculation labelled with a circle.}
	\label{fig:osculation}
\end{figure}
\begin{proof}[Proof of Theorem \ref{thm:DIinversions}]
	Fix an $n\times n$ alternating sign matrix which corresponds to a Dyck island with boundary paths $\gamma_1, \ldots , \gamma_\ell$. 
	
	If $\gamma_1, \ldots, \gamma_\ell$ are disjoint and not nested, the formula is clear, since we can consider a block diagonal decomposition of smaller alternating sign matrices, each containing a single $\gamma_i$. Then the result follows by applying Lemma \ref{lem:singleLoop}. 

	To deal the case when some of the $\gamma_i$ are nested, but have no off-diagonal osculations, we claim that we can evaluate the contributions of pairs containing vertices in the southwest boundaries and northeast boundaries of the $\gamma_i$ according to the formulas:
	\begin{align*}
		\C(\swb (\gamma)) &= \semilength(\gamma) + \nesting(\gamma) (1 + \#\{\mathcal{O}^\gamma_{sw}\}) \\
		\C(\neb (\gamma)) &= \nesting(\gamma) (1+ \#\{\mathcal{O}^\gamma \setminus \mathcal{O}^\gamma_{sw}\})
	\end{align*}
	A justification of these formulas will be provided in Lemma \ref{lem:swbneb}.
	
	Then if $\delta$ is the Dyck island corresponding described by the boundary paths $\gamma_1 , \ldots , \gamma_\ell$ with diagonal ones labeled $p_1, \ldots, p_m$ and with no off-diagonal osculations, we have the following evaluation of $\inv(\delta)$:
	\begin{align*}
		\inv(\delta) &= \sum_{i=1}^\ell \C(\swb(\gamma_i)) + \C(\neb(\gamma_i)) + \sum_{j=1}^m \nesting(p_j) \\
		&= \sum_{i=1}^\ell \semilength(\gamma_i) + \nesting(\gamma_i)(2+ \#\{\mathcal{O}^{\gamma_i}\}) + \sum_{j=1}^m \nesting(p_j) \\
		&= \sum_{i=1}^\ell \semilength(\gamma_i) + \diag(\gamma_i) \\
		&= \sum_{i=1}^\ell \inv(\gamma_i).
	\end{align*}
	The second to last equality holds because the number of diagonal osculations exactly account for the diagonal ones that would have been there otherwise. 
	Also, the value $2$ accounts for the two diagonal ones which would have been in place of the left and right endpoints of the northeast and southwest boundaries of $\gamma_i$. This is equivalent to the assertion that 
	\begin{equation*}
		\sum_{i=1}^\ell N(\gamma_i)(2+\#\{\mathcal{O}^{\gamma_i}\}) + \sum_{j=1}^m \nesting(p_j) = \sum_{i=1}^\ell \diag(\gamma_i).
	\end{equation*}
	See figures \ref{fig:calculation1} and \ref{fig:calculation2} for specific examples of this calculation.

	Lastly, let us account for off-diagonal osculations. Suppose $v_1, \ldots, v_k$ are all of the off-diagonal osculations in a given Dyck island. By definition, for each $v_i$, there are two distinct boundary paths $\gamma_{i_1}$ and $\gamma_{i_2}$ which meet at $v_i$. Without loss of generality, suppose that $\neb(\gamma_{i_1})$ is incident to $v_i$ in a vertex of type $ud$ and that $\neb(\gamma_{i_2})$ is incident to $v_i$ in a vertex of type $du$. It must follow that $\gamma_{i_1}$ is contained in the interior of $\gamma_{i_2}$. 
	Then suppose $v_i$ splits into two vertices $v_{i_1}$ and $v_{i_2}$ located at the same lattice point, with the formal condition that $v_{i_2}$ is considered above and to the right of $v_{i_1}$ and that $v_{i_1}$ is in $\neb(\gamma_{i_1})$ and $v_{i_2}$ is in $\neb(\gamma_{i_2})$. For the sake of computing the inversion number, we assume that $v_{i_1}$ corresponds to a $1$ and that $v_{i_2}$ corresponds to $-1$. If we carry out this procedure for each $i$ from $1$ to $k$, then we have formally reduced to the case above with no osculations. To finish, we simply need to calculate the effect of ``merging'' the two vertices $v_{i_1}$ and $v_{i_2}$ back into the vertex $v_i$. Suppose that $v_{i_2}$ contributes $-x$ to $\C(\neb(\gamma_{i_2}))$ for some non-negative integer $x$. It follows that $v_{i_1}$ contributes $x+1$ to $\C(\neb(\gamma_{i_1}))$ since it is nested one additional level beyond $v_{i_2}$ since it is (formally) in the interior of $\gamma_{i_2}$. The alternating sign matrix entry located at $v_i$ corresponds to a $0$, therefore the effect of merging $v_{i_1}$ and $v_{i_2}$ to the inversion number sum is to negate the contributions $-x + x +1$. Note that this splitting and merging procedure does not influence the contributions of other vertices to the inversion number sum.
	 Thus in the end, the sum $\sum_{i=1}^\ell \inv(\gamma_i)$ overestimates the inversion number by exactly the number of off-diagonal osculations, $k$, and the equation holds.
\end{proof}

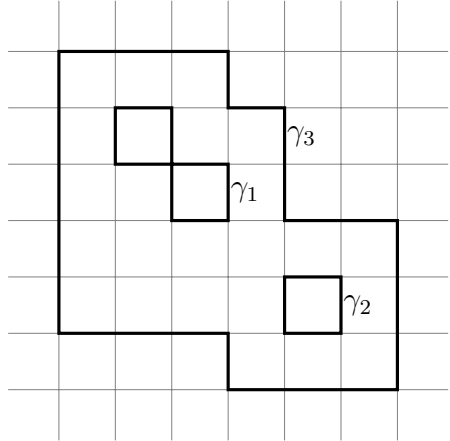
\begin{figure}[h]
	\begin{center}
		\begin{tikzpicture}[scale=.75]
				\draw[very thin, gray] (-.9,-.9) grid (6.9,6.9);
				\draw[very thick] (0,6) -- (3,6) -- (3,5) -- (4,5) -- (4,3) -- (6,3) -- (6,3) -- (6,0) -- (3,0) -- (3,1) -- (0,1) -- cycle;
				\draw[very thick] (4,1) rectangle (5,2);
				\draw[very thick] (1,5) -- (2,5) -- (2,4) -- (3,4) -- (3,3) -- (2,3) -- (2,4) -- (1,4) -- cycle;
				\node at (3.3,3.5) {$\tiny{\gamma_1}$};
				\node at (5.3,1.5) {$\tiny{\gamma_2}$};
				\node at (4.3,4.5) {$\tiny{\gamma_3}$};
		\end{tikzpicture}
	\end{center}
	\caption{We calculate the inversion number of the given Dyck island $\delta$ in two ways. First, we use the formula of Theorem \ref{thm:DIinversions}: $\inv(\delta) = \inv(\gamma_1) + \inv(\gamma_2) + \inv(\gamma_3)= 2 + 1 + 11 = 14$. Next, we calculate by summing contributions (note that there are no diagonal ones): $\sum_{i=1}^3 \C(\swb(\gamma_i)) + \C(\neb(\gamma_i)) = 14$ since $\C(\swb(\gamma_1)) = 4$, $\C(\neb(\gamma_1)) = 1$, $\C(\swb(\gamma_2)) = 2$, $\C(\swb(\gamma_2)) = 1$, $\C(\swb(\gamma_3)) = 6$, and $\C(\neb(\gamma_3)) = 0$.}
	\label{fig:calculation1}
\end{figure}
\begin{figure}[h]
	\begin{center}
		\begin{tikzpicture}[scale=.75]
				\draw[very thin, gray] (-.9,-.9) grid (6.9,6.9);
				\draw[very thick] (0,6) -- (3,6) -- (3,5) -- (4,5) -- (4,3) -- (6,3) -- (6,3) -- (6,0) -- (3,0) -- (3,1) -- (0,1) -- cycle;
				\draw[very thick] (4,1) rectangle (5,2);
				\draw[very thick] (1,5) -- (2,5) -- (2,4) -- (3,4) -- (3,3) -- (1,3) -- cycle;
				\node at (3.3,3.5) {$\tiny{\gamma_1}$};
				\node at (5.3,1.5) {$\tiny{\gamma_2}$};
				\node at (4.3,4.5) {$\tiny{\gamma_3}$};
		\end{tikzpicture}
	\end{center}
	\caption{For comparison, we modify $\gamma_1$ so that only $\neb(\gamma_1)$ touches the diagonal. Again, we use the formula of Theorem \ref{thm:DIinversions}: $\inv(\delta) = \inv(\gamma_1) + \inv(\gamma_2) + \inv(\gamma_3)= 2 + 1 + 11 = 14$. Next, we calculate by summing contributions noting again that there are no diagonal ones: $\sum_{i=1}^3 \C(\swb(\gamma_i)) + \C(\neb(\gamma_i)) = 14$ since $\C(\swb(\gamma_1)) = 3$, $\C(\neb(\gamma_1)) = 2$, $\C(\swb(\gamma_2)) = 2$, $\C(\swb(\gamma_2)) = 1$, $\C(\swb(\gamma_3)) = 6$, and $\C(\neb(\gamma_3)) = 0$.}
	\label{fig:calculation2}
\end{figure}
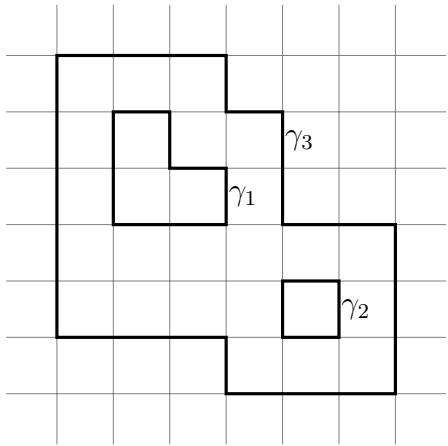

\begin{lemma} 
Consider an $n \times n$ alternating sign matrix corresponding to a Dyck island described by boundary paths $\gamma_1,\ldots , \gamma_\ell$. The contributions from the southwest boundaries and northeast boundaries of each $\gamma_i$ are:
	\begin{align*}
		\C(\swb (\gamma_i)) &= \semilength(\gamma_i) + \nesting(\gamma_i) (1 + \#\{\mathcal{O}^{\gamma_i}_{sw}\}) \\
		\C(\neb (\gamma_i)) &= \nesting(\gamma_i) (1+ \#\{\mathcal{O}^{\gamma_i} \setminus \mathcal{O}^{\gamma_i}_{sw}\}).
	\end{align*}
	\label{lem:swbneb}
\end{lemma}
\begin{proof}
	Suppose that $\gamma_i$ is in the interior of N loops. First, let us prove the northeast boundary equation. Consider the Dyck word representing $\neb(\gamma_i)$. Generically, each `$ud$' contributes $+\mathrm{N}$ and each `$du$' contributes $-\mathrm{N}$. Thus by Lemma \ref{lem:words}, the total contribution is $\mathrm{N}$. The only exception to this rule is when $\neb(\gamma_i)$ touches the diagonal at a vertex $v$ such that $v \notin\mathcal{O}^\gamma_{sw}$. In this case, the corresponding `$du$' does not contribute a $-\mathrm{N}$ since there is a $0$ in place of a $1$ in the alternating sign matrix picture. This establishes the equation for the contribution to the northeast boundary.

	For the southwest boundary, $\swb(\gamma_i)$,  consider the corresponding Dyck word. Generically, each instance of `$ud$' contributes ($\mathrm{N} + \height$) and each instance of `$du$' contributes ($-\mathrm{N} - \height$). By Lemma \ref{lem:words} and the calculation in Lemma \ref{lem:singleLoop}, this word evaluates to $\semilength(\gamma) + \mathrm{N}$. There are two special cases to consider: corner vertices of $\swb(\gamma_i)$ along the diagonal which are either in $\neb(\gamma_i)$ or not in $\neb(\gamma_i)$. When a vertex along the diagonal is also a corner vertex of $\swb(\gamma_i)$ but not $\neb(\gamma_i)$, it corresponds to a $0$ in the alternating sign matrix picture instead of the $-1$ in the generic case. Thus we compensate by adding an extra factor of $N$ to $\C(\swb(\gamma_i))$. In the other case when a given vertex along the diagonal is simultaneously in $\swb(\gamma_i)$ and $\neb(\gamma_i)$, we observe that this vertex corresponds to $-1$ in the alternating sign matrix picture, but that it has already been accounted for in $\C(\neb(\gamma_i))$. Hence, every diagonal vertex on $\swb(\gamma_i)$ will contribute $0$ to $\C(\swb(\gamma_i))$ instead of $-\mathrm{N}$. Thus the equation for the contribution to the southeast boundary is established.
\end{proof}
\subsection*{Acknowledgements}
The author gratefully acknowledges Jessica Striker and anonymous referees for helpful observations and suggestions for improvement.

\bibliographystyle{amsplain}
\bibliography{di}

\end{document}